\documentclass{amsart}
\usepackage{amssymb}  
\usepackage{amsthm} 
\usepackage{amsmath}
\usepackage{amsfonts}
\usepackage{graphicx}
\usepackage[utf8]{inputenc}
\usepackage[all]{xy}
\usepackage{placeins}
\usepackage[linktocpage=true,colorlinks=true, allcolors=blue]{hyperref}
\newtheorem{theorem}{Theorem}
\newtheorem{lemma}{Lemma}
\newtheorem{corollary}{Corollary}
\newtheorem{proposition}{Proposition}\theoremstyle{definition}
\newtheorem{example}{Example}
\author{José Juan-Zacarías}
\address{Instituto de Matemáticas Unidad
 Cuernavaca, Av. Universidad s/n. Col. Lomas de Chamilpa Código
 Postal 62210, Cuernavaca, Morelos.}
\email{jose.juan@im.unam.mx} 
\date{October 19, 2018} 
\thanks{\emph{2010 Mathematics Subject Classification:}
30F10, 30D30, 14H52.
\\
\indent \emph{Keywords and phrases:}
Elliptic curves, cross ratios, Hesse normal form.}

\begin{document} 
\title[Elliptic curves and four points]{Some remarks on the 
correspondence between elliptic curves and four points in the Riemann 
sphere} 

\maketitle

\begin{abstract}
 In this paper we relate some classical normal forms for complex 
 elliptic curves in terms of 4-point sets in the Riemann sphere. Our 
 main result is an alternative proof that every elliptic curve is 
 isomorphic as a Riemann surface to one in the Hesse normal form. In 
 this setting, we give an alternative proof of the equivalence betweeen 
 the Edwards and the Jacobi normal forms. Also, we give a geometric 
 construction of the cross ratios for 4-point sets in general position. 
\end{abstract}

\section*{Introduction} A complex elliptic curve is by definition a 
compact Riemann surface of genus 1. By the \emph{uniformization 
theorem}, every elliptic curve is conformally equivalent to an 
algebraic curve given by a cubic polynomial in the form

\begin{equation}\label{Weierstrass-intro}
 E\colon \  y^2=4x^3-g_2x-g_3, \quad \text{with}\ 
 \Delta=g_2^3-27g_3^2\neq 0,
\end{equation}

\noindent this is called the \emph{Weierstrass normal form}. For 
computational or geometric purposes it may be necessary to find a 
Weierstrass normal form for an elliptic curve, which could have been 
given by another equation. At best, we could predict the right changes 
of variables in order to transform such equation into a Weierstrass 
normal form, but in general this is a difficult process.

\par A different method to find the normal form 
(\ref{Weierstrass-intro}) for a given elliptic curve, avoiding any 
change of variables, requires a degree 2 meromorphic function on the 
elliptic curve, which by a classical theorem always exists and in many 
cases it is not difficult to compute. By the Riemann-Hurwitz formula, 
this map has four branch points and, by composing with a Möbius 
transformation, we can assume that one point is at $\infty$ and the 
other points are three complex numbers $\{z_{1},z_{2},z_{3}\}$. Let $C$ 
be the centroid of the points $z_{i}$ and $e_{i}=z_{i}-C$. With this 
points we are able to find the constants $g_{2}$ and $g_{3}$ of 
(\ref{Weierstrass-intro}) as the coefficients of the polynomial

\begin{equation}
	4(x-e_{1})(x-e_{2})(x-e_{3}).
\end{equation}  

\par This motivates the study of another classical normal forms in this 
setting. We are interested in relating the branch points of degree 2 
meromorphic functions on the elliptic curves with 4-point 
configurations derived from a given normal form. This relation 
naturally leads us to the study of the cross ratio of 4-point sets. See 
Example \ref{example-Fermat} and Example \ref{example-Steinmetz} for 
illustrate the above discussion. 

\par Our  main result is an alternative proof that every elliptic curve 
is isomorphic as a Riemann surface to one in the Hesse normal form (see 
Theorem \ref{Theorem-Hesse}). In this setting, we give an alternative 
proof of the equivalence between the Edwards and the Jacobi normal 
forms (see Theorem \ref{Jacobi-Edwards-equivalence}). We give a 
geometric construction of the cross ratios for 4-point sets in general 
position in terms of the angles of certain curvilinear triangles (see 
Figure \ref{curvilinear-triangles} and Figure 
\ref{three-complex-points}). We don't know whether this construction 
was known earlier. 
 
\par This paper is organized as follows: in  Section 
\ref{Section-correspondence} we explain briefly the correspondence 
between elliptic curves and 4-point sets in the Riemann sphere. Also we 
give some remarks on the cross ratio of 4-point sets. In Section 
\ref{Section-Legendre-Weierstrass} we describe the configuration 
associated to the Legendre and the Weierstrass normal forms. In Section 
\ref{Section-Jacobi-Edwards} we describe the configuration of points 
for the Jacobi and the Edwards normal forms and how they are related. 
Finally, in Section \ref{Section-Hesse} we prove the main theorem.

\section{Correspondence between $\mathcal{M}$ and four 
points}\label{Section-correspondence}

Recall that an elliptic curve is a Riemann surface of genus one. We 
denote by $\mathcal{M}$ the \emph{moduli of elliptic curves}, i.e., the 
set of classes of isomorphism of elliptic curves, where by isomorphism 
we mean a biholomorphism between Riemann surfaces.

Every elliptic curve admits a meromorphic function of degree 2 and by 
the Riemann-Hurwitz formula this map has four critical values. 
Reciprocally, given a 4-point set in the Riemann sphere there is a 
degree 2 meromorphic function with critical values on those points. 
The above correspondence is summed in following theorem (see \cite[\S 
6.3.2]{Donaldson}): 

\begin{theorem}\label{the-correspondence}
 There is a bijection between $\mathcal{M}$ and 4-point sets on the 
 Riemann sphere modulo Möbius transformations. 
\end{theorem}

\par \noindent \textit{Remark.} In this paper we assume that the
complex plane $\mathbb{C}$ has the canonical orientation, the positive 
orientation is the counterclockwise orientation.

\par Now we will discuss a geometric interpretation of the cross ratio 
of 4-point sets when one of them is at $\infty$ and the other three 
complex numbers are not collinear. 

\subsection{The cross ratio and three points in the complex 
plane.}\label{triangle-shapes} Recall that the cross ratio 
$\chi(z_1,z_2,z_3,z_4)$ of four distinct and ordered points 
$(z_1,z_2,z_3,z_4)$ in the Riemann sphere, is the value $\mu(z_4)$ 
where $\mu$ is the unique Möbius transformation which sends orderly the 
three points $(z_1,z_2,z_3)$  to $(0,1,\infty)$. When the four points 
are complex numbers the cross ratio becomes:

\begin{equation}\label{cross-ratio}
 \chi(z_1,z_2,z_3,z_4)=\frac{(z_4-z_1)(z_2-z_3)}{(z_1-z_2)(z_3-z_4)}.
\end{equation}

\par We begin with this elementary remark:

\begin{lemma}\label{elementary-remark}
 Suppose that $\{z_1,z_2,z_3, \infty\}$ and $\{w_1,w_2,w_3,\infty\}$ 
 are two 4-point sets in the Riemann sphere which are Möbius 
 equivalent, then there exists an affine transformation $az+b$ that 
 maps one set to the other.
\end{lemma}
\begin{proof}
 Suppose that $\mu(z)$ is the Möbius transformation which sends 
 $\{z_1,z_2,z_3, \infty\}$ to $\{w_1,w_2,w_3,\infty\}$. If $\mu$ fixes 
 $\infty$ there is nothing to do. If $\mu(\infty)=w_{i_0}$, we can take 
 the Möbius transformation $T$ that interchanges $\infty$ with 
 $w_{i_0}$ and interchange the other points. Such transformation exists 
 because this kind of permutations do not change the cross ratio. Then 
 the affine transformation is $T\circ \mu$.
\end{proof}

The Lemma \ref{elementary-remark} has interesting consequences. First, 
since we can send every point to $\infty$ by a Möbius transformation it 
follows directly from Lemma \ref{elementary-remark} and Theorem 
\ref{the-correspondence} the corollary:

\begin{corollary}\label{correspondence-triangles}
 There is a bijection between $\mathcal{M}$ and the sets of three 
 points on the complex plane $\mathbb{C}$ modulo complex affine 
 transformations.
\end{corollary}

\par \noindent \emph{Remark.} In what follows we assume that our three 
points are not collinear; also we suppose that the triangle formed by 
these points is positively oriented.\\

\par We say that two sets of three points in the complex plane are 
\emph{equivalent} if there is an affine map $az+b$ which maps one set 
to the other. Recall that, by definition, two triangles are similar if 
they have the same angles. And two (oriented) triangles are similar if 
and only if there is an affine transformation $az+b$, with $a,b\in 
\mathbb{C}$ and $a\neq 0$, which maps one triangle to the other. It 
follows from the Corollary \ref{correspondence-triangles} that two sets 
of three points are equivalent if and only if their (oriented) 
triangles are similar.

\par From the above discussion follows the proposition: 

\begin{proposition}
	Let $z_{1}$, $z_{2}$ $z_{3}$ be three non collinear points in the 
	complex plane. The number $\lambda$ is a cross ratios for the 
	4-point set $\{z_1,z_2,z_3,\infty\}$ if and only if the oriented 
	triangle $\{0,1,\lambda\}$ is similar to the oriented triangle 
	$\{z_1,z_2,z_3\}$. Therefore there are generically six equivalent 
	cross ratios of each 4-point set $\{z_1,z_2,z_3,\infty\}$ (see 
	Figure \ref{similar-triangles}).
\end{proposition}
	
\begin{figure}
 \begin{center}
  \includegraphics[width=10cm]{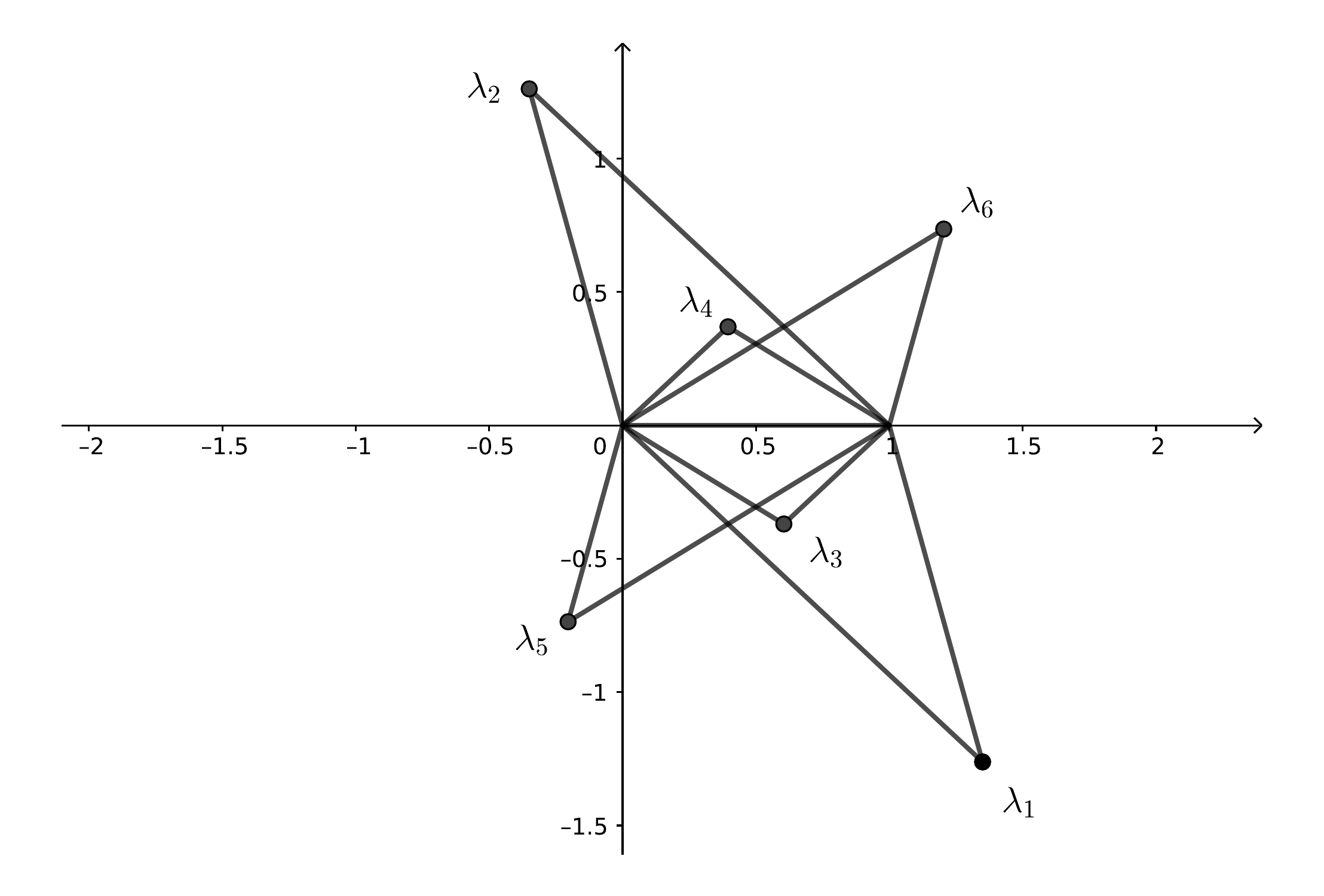}
 \caption{Six similar triangles representing equivalent cross
ratios.}
\label{similar-triangles}
 \end{center}
\end{figure}

\par We will see in Section \ref{Section-Hesse} that for not collinear 
points the only exception is when the oriented triangle 
$\{z_1,z_2,z_3\}$ is equilateral. If $\lambda$ is a cross ratio for 
some 4-point set, the other values are given by 
(\ref{equivalent-cross-ratios}).

\par The proof of Lemma \ref{elementary-remark} tells us how to 
construct a cross ratio for a particular order:

\begin{proposition}
	Let $z_{1}$, $z_{2}$ $z_{3}$ be three non collinear points in the 
	complex plane forming an oriented triangle with respective angles 
	$\alpha,\beta,\gamma$. The cross ratios 
	$\lambda_1=\chi(z_1,z_2,z_3,\infty)$ and 
	$\lambda_2=\chi(z_1,z_2,\infty,z_3)$ can be obtained by the 
	geometric construction illustrated in Figure 
	\ref{cross-ratio-ordered}.
\end{proposition} 

\begin{figure}
\begin{center}
  \includegraphics[width=12cm]{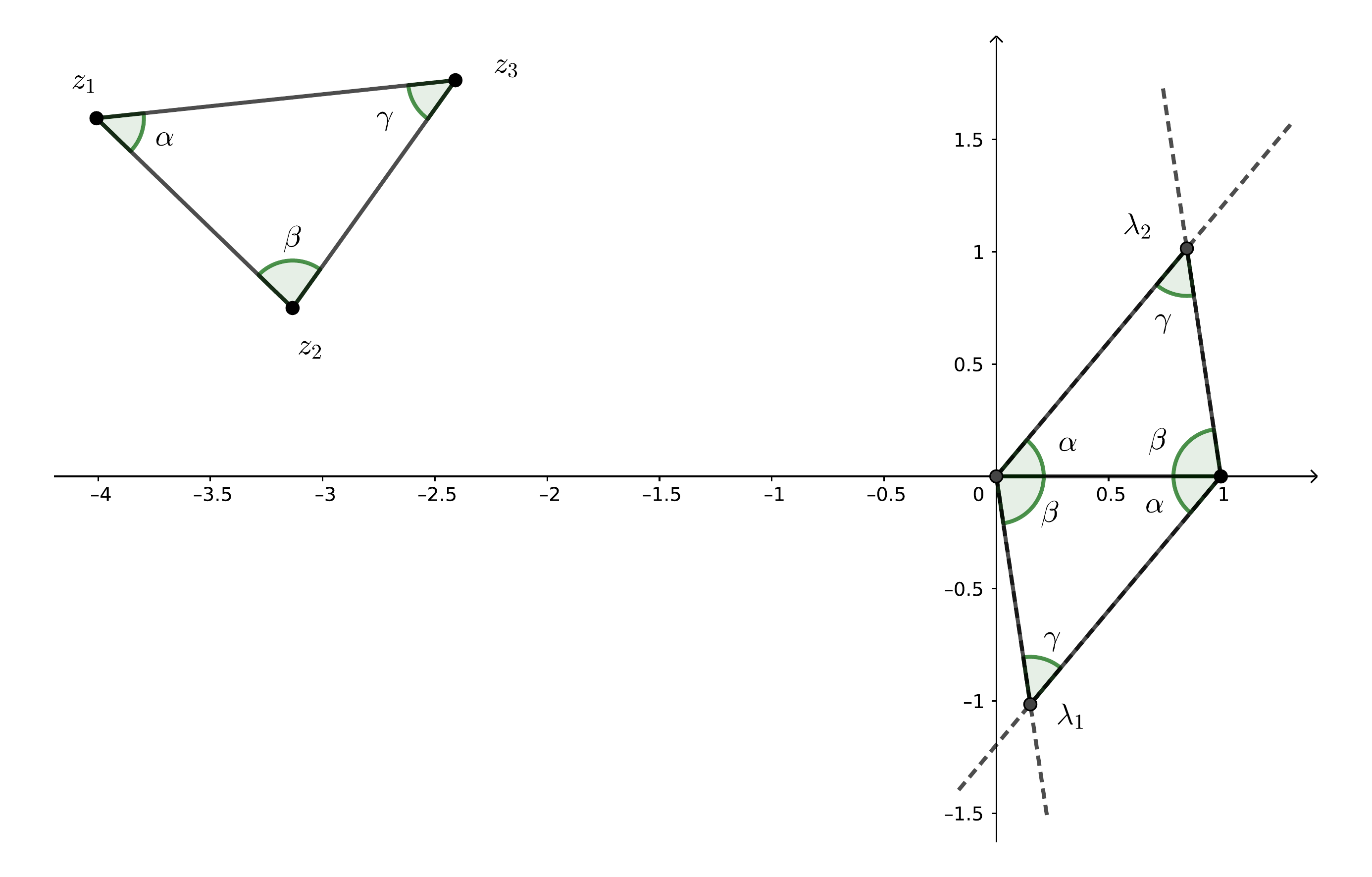}
\caption{Geometric construction of
$\lambda_1=\chi(z_1,z_2,z_3,\infty)$ and
$\lambda_2=\chi(z_1,z_2,\infty,z_3)$.}
\label{cross-ratio-ordered} \end{center} \end{figure}

\par In order to find geometrically the cross ratio for a 4-point set 
in general position in the complex plane we need to find the angles of 
the curvilinear triangles of we call the \emph{shape} of a 4-point set, 
which we describe bellow.

\subsection{The 4-point shape.}\label{shape-four-points} In what 
follows we assume that our four points are in general position, i.e. 
they not lie in a same circle.

\par Let $\{z_1,z_2,z_3,z_4\}$ be a 4-point set in the complex plane, 
for each three points there are exactly one circle passing through 
them, then there are four circles associated to these points. Consider 
the curvilinear triangles formed by arcs of these circles whose 
vertices are in each set of three points of $\{z_1,z_2,z_3,z_4\}$ (see 
Figure \ref{curvilinear-triangles}), we suppose that they are oriented 
(with the positive orientation of $\mathbb{C})$.

\begin{figure}
	\begin{center}
	\includegraphics[width=12cm]{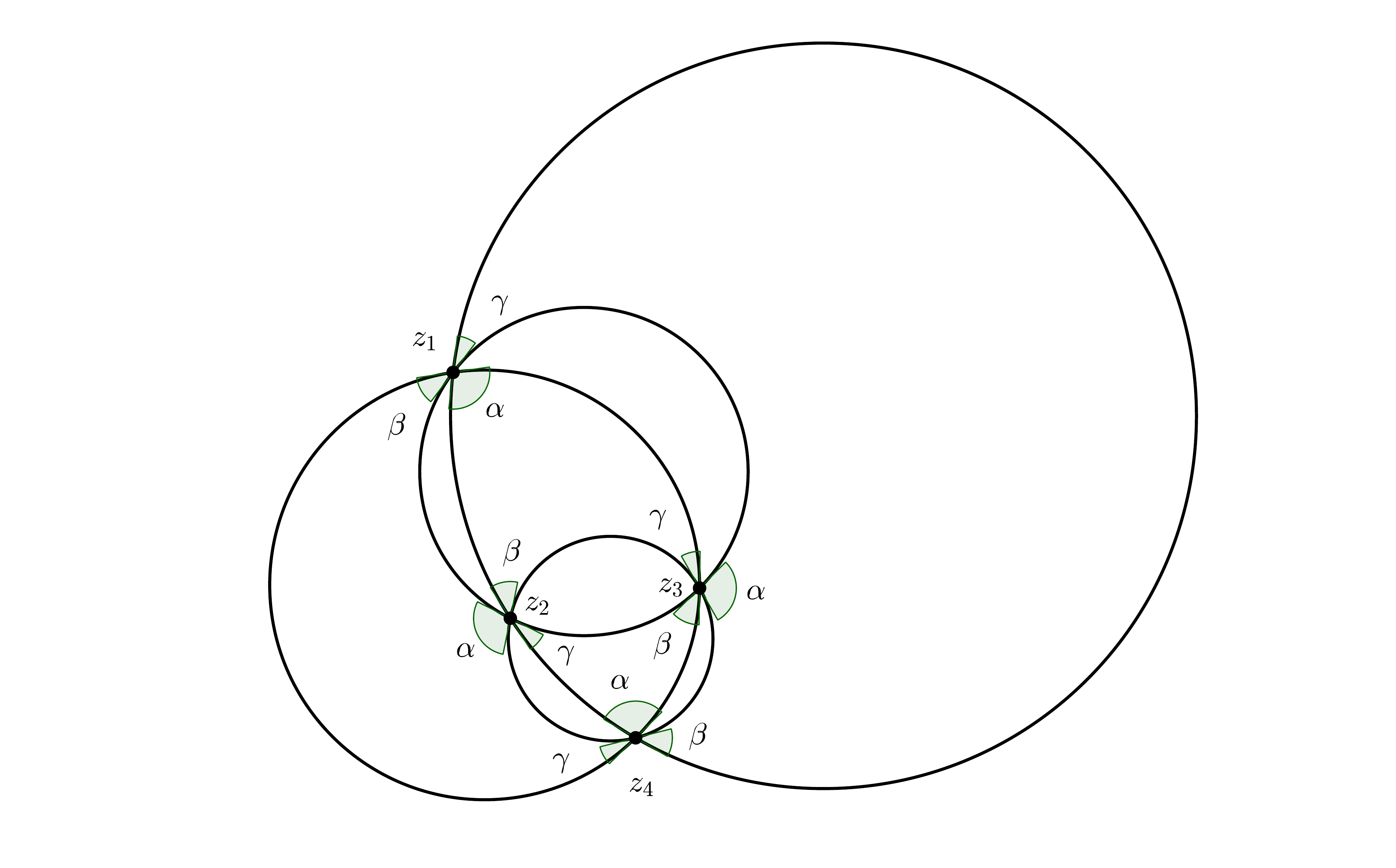}
	\caption{Curvilinear triangles of four points in the complex plane.}
	\label{curvilinear-triangles} 
	\end{center} 
\end{figure}

These four curvilinear triangles have the same angles: to see this, 
observe, like in the Lemma \ref{elementary-remark} that the 
permutations of kind 2+2, i.e, product of two transpositions do not 
change the cross ratio, then there exist a Möbius transformation which 
permutes the four points interchanging a pair of points $z_i$ with 
$z_j$ and interchanging the remainder two. This kind of transformations 
form a group isomorphic to the Klein group. As Möbius transformations 
maps circles to circles conformally, preserving the orientation, this 
group acts transitively in the four triangles. 

\par From the last argument also follows that two equivalent 4-point 
sets in the complex plane have oriented curvilinear triangles with the 
same angles. 

\par If we allow \emph{generalized circles} the same is true when one 
point is $\infty$ and the other are non collinear complex numbers, in 
this case one curvilinear triangle is an euclidean triangle, the other 
triangles have an arc and two lines as its sides (see Figure 
\ref{three-complex-points}). The euclidean triangle is positively 
oriented, which induce an orientation to the other triangles. 

\par Reciprocally, if their oriented curvilinear triangles have the 
same angles they are equivalent, we can see it sending one point to 
$\infty$ and considering the previous case of euclidean triangles. So 
we can define the \emph{shape} of a 4-point set in the Riemann sphere 
as the oriented curvilinear triangles obtained in the before 
construction. Then we have the following theorem:

\begin{theorem}\label{shape-theorem}
	Two 4-point set in general position in the Riemann sphere 
	are Möbius equivalent if and only if their shapes have the same 
	angles (see Figure \ref{curvilinear-triangles} and Figure 
	\ref{three-complex-points}).
\end{theorem}

\begin{figure}
\begin{center}
 \includegraphics[width=10cm]{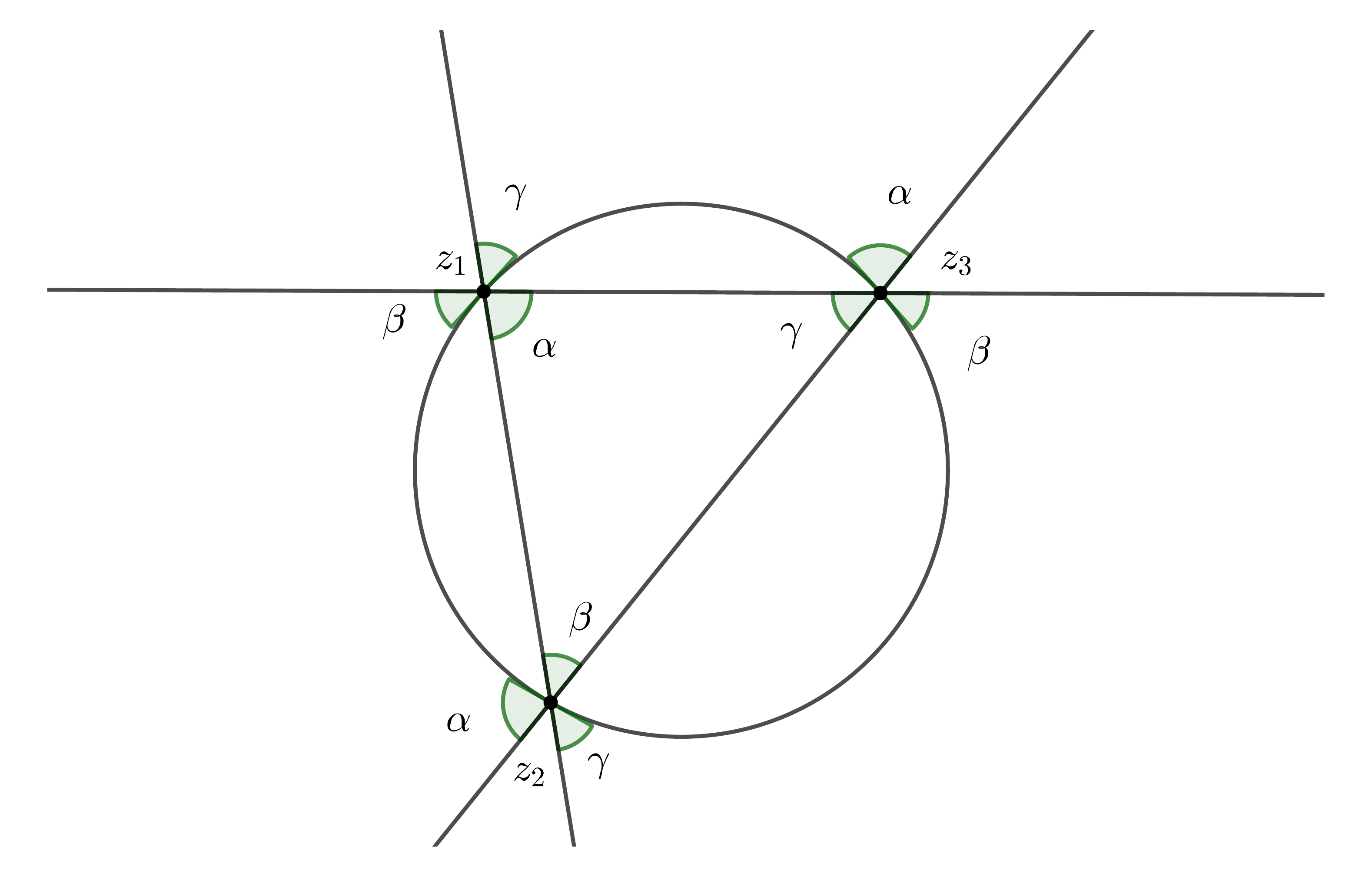}
 \caption{Curvilinear triangles of four points in the Riemann sphere
 with one point at $\infty$.}
\label{three-complex-points}
\end{center}
\end{figure} 

Then, the cross ratios of four points in the complex plane can be 
found geometrically analogously to \ref{triangle-shapes}, we exemplify 
the construction of two of them:

\begin{proposition}
	Let $\{z_1,z_2,z_3,z_4\}$ four points in the complex plane in 
	general position. Consider the oriented curvilinear triangle with 
	vertices in $z_{1}$, $z_{2}$ and $z_{3}$ with respective angles 
	$\alpha,\beta,\gamma$. The cross ratios 
	$\lambda_1=\chi(z_1,z_2,z_3,z_4)$ and 
	$\lambda_2=\chi(z_1,z_2,z_4,z_3)$ can be obtained by the geometric 
	construction illustrated in Figure 
	\ref{cross-ratio-ordered-finite}.
\end{proposition}

\begin{figure}
\begin{center}
  \includegraphics[width=12cm]{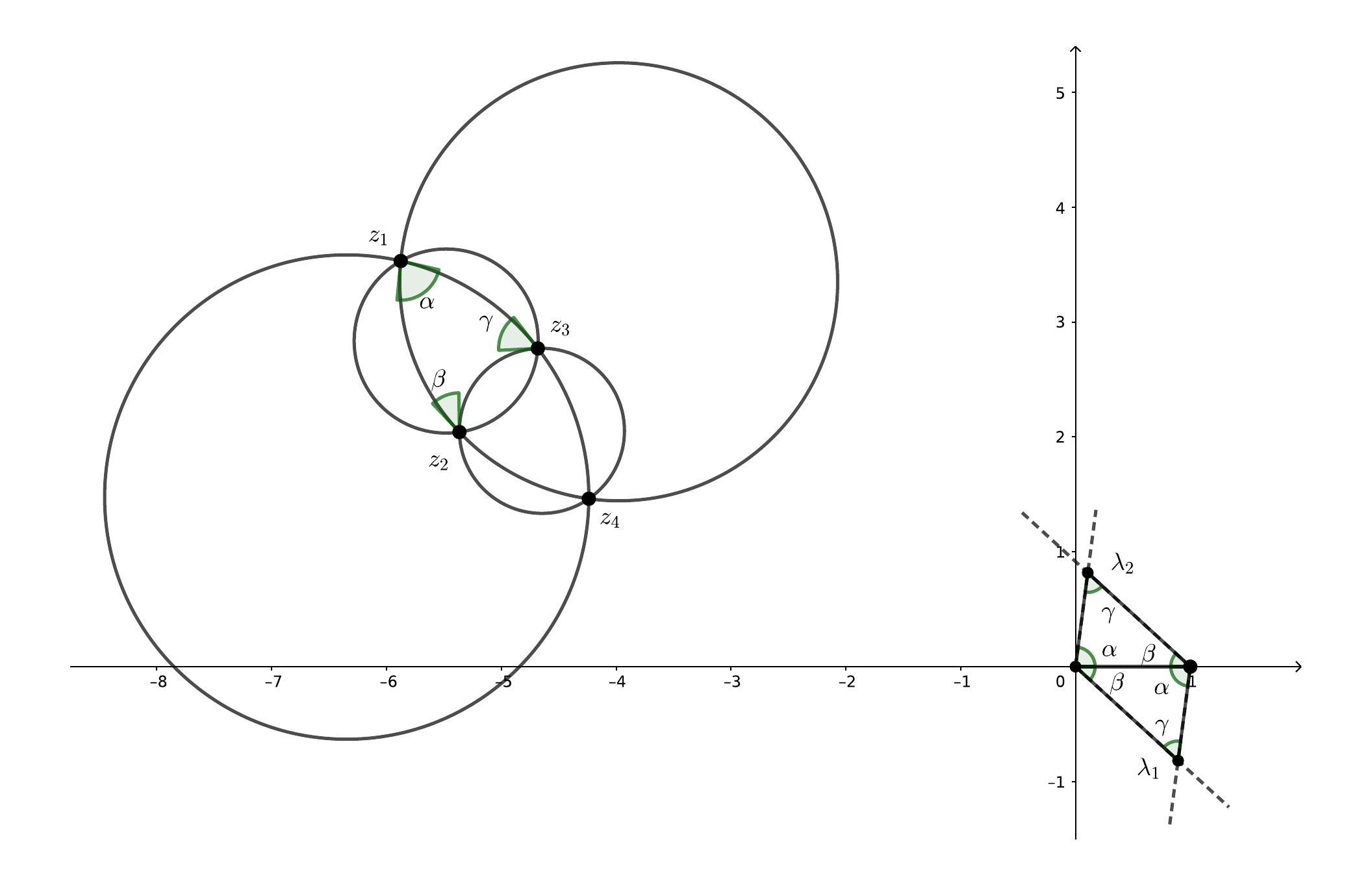}
\caption{Geometric construction of $\lambda_1=\chi(z_1,z_2,z_3,z_4)$
and $\lambda_2=\chi(z_1,z_2,z_4,z_3)$.}
\label{cross-ratio-ordered-finite} 
\end{center} 
\end{figure}

\section{Legendre and Weierstrass normal 
form.}\label{Section-Legendre-Weierstrass} Recall that if $p(x)$ is a 
cubic polynomial with complex coefficients and different roots, the 
Riemann surface $E\colon\ y^2=p(x)$ has the projection to the first 
coordinate as a meromorphic function of degree two which ramifies 
exactly in the roots of $p(x)$ and $\infty$ (see \cite[8.10]{Forster}), 
then, by the Corollary \ref{correspondence-triangles} every elliptic 
curve is isomorphic to one elliptic curve of that form. By the 
discussion in the previous section, two such elliptic curves $E_1\colon 
\ y^2=p(x)$, $E_2\colon \ y^2=q(x)$  are isomorphic if and only if the 
roots of $p(x)$ and $q(x)$ are equivalent or equivalently if their 
oriented triangles have the same angles, for non-degenerated triangles.

\par Given four points we can also take some cross ratio of them, so
we can take the elliptic curve as
\begin{equation}
 E\colon \quad y^2=x(x-1)(x-\lambda),\quad
 \lambda\in\mathbb{C}-\{0,1\}.
\end{equation}

This is  called a Legendre normal form of an elliptic curve. Two such 
elliptic curves $E_1\colon \ y^2=x(x-1)(x-\lambda_1)$, $E_2\colon \ 
y^2=x(x-1)(x-\lambda_2)$ are isomorphic if and only if the $\lambda_1$ 
and $\lambda_2$ are equivalent,  we explained in \ref{triangle-shapes} 
how to read geometrically into this equivalence: they are equivalent if 
and only if the oriented triangles with vertices at $0,1,\lambda_{1}$ 
and $0,1,\lambda_{2}$ are similar.

\par Another normal form is obtained considering the centroid $C$ of
three points $\{z_1,z_2,z_3\}$ in the complex plane to obtain three
points $e_i=z_i-C$ whose centroid is at 0, so the polynomial
$4(x-e_1)(x-e_2)(x-e_3)$ has the form $4x^3-g_2x-g_3$, then the
elliptic curve is given by:

\begin{equation}
 E\colon\quad  y^2=4x^3-g_2x-g_3, \quad \Delta=g_2^3-27g_3^2\neq 0,
\end{equation} where $\Delta=g_2^3-27g_3^2$ is the discriminant of the 
cubic polynomial. Recall $\Delta\neq 0$ if and only if the cubic 
polynomial has distinct roots. This kind of elliptic curves are called 
a Weierstrass normal form. Thus we can interpret geometrically the 
Weierstrass normal form as all the sets of  three points in the complex 
plane whose centroid is at $0$. As we know the centroid can be found 
geometrically, as the intersection of the three medians. As affine 
transformations preserve centroids, we have that two elliptic curves in 
the Weierstrass normal form are isomorphic if and only if the roots  of 
the polynomials differ by a complex factor.

\subsection{Some examples} To illustrate the above discussion we
give a couple of examples. First we compute a Weierstrass and a
Legendre normal form of the Fermat cubic in this setting:

\begin{example}\label{example-Fermat}
 The Fermat cubic is an elliptic curve given by $x^3+y^3=1$. We can
 take for example the meromorphic function $x+y$, it is of degree 2
 with critical values 0 and the cubic roots of 4, sending 0 to
 $\infty$ by $1/z$, we obtain that the Fermat cubic is the
 equilateral triangle with vertices in the cubic roots of $1/4$. In
 this case the centroid is already in the origin, then  the
 Weierstrass normal form is $y^2=4x^3-1$. Denote by $\rho$ the cubic
 root $\exp(2\pi i/3)$. To obtain the cross ratio hang the
 equilateral triangle on the interval $[0,1]$, the vertices are
 $\{0,1,-\rho\}$. So the Legendre normal form is
 $y^2=x(x-1)(x+\rho)$.
\end{example}

\par  For the next example we need elliptic curves given by quartic
equations. Given a quartic polynomial $q(x)$ with complex coefficients
and distinct roots, the Riemann surface $E\colon \ y^2=q(x)$ has the
projection to the first coordinate as a meromorphic function of
degree 2 whose critical values are exactly the roots of $q(x)$ (see
\cite[8.10]{Forster}). Note that the  curve $E$ has not smooth
projective closure in $\mathbb{CP}^2$.

\begin{example}\label{example-Steinmetz}
 In \cite[\S 2.3.5]{Steinmetz} is considered the following family of
 algebraic curves: $F\colon\quad x^n+y^m=1\quad (n\geq m\geq 2)$. For
 the cases (4,2), (3,3) and (3,2) they are the elliptic curves:
 
 \begin{equation*}
  x^4+y^2=1,\quad x^3+y^3=1,\quad x^3+y^2=1.
 \end{equation*}

 With our method we can see that the last two are isomorphic between
 them and non-isomorphic to the first. For that, consider the
 projection to the first coordinate on the curve $(n,m)=(4,2)$, it
 ramifies on the fourth roots of 1, as they are in the same circle,
 the cross ratio is real. The second case was calculated in the
 previous example, it corresponds to equilateral triangle. In the
 third case the projection to the first coordinate ramifies on the
 cubic roots of unity and $\infty$. By the previous discussion it
 follows the affirmation.
\end{example}

\section{The Jacobi and Edwards normal 
form}\label{Section-Jacobi-Edwards} Recall that every 4-point set
in the real line can be mapped by some Möbius function to four 
symmetric points respect to zero $\{a,-1/a,-a,1/a\}$. In fact, this can be done with every four points in the complex plane. 
Because, for $a\in \mathbb{C}-\{0\}$ such that $a^4\neq 1$, the cross 
ratio

\begin{equation}
 \varphi(a)=\chi\left(a,-\frac{1}{a},-a,\frac{1}{a}\right)=
 \left(\frac{1-a^2}{1+a^2}\right)^2
\end{equation}

\noindent is well defined, and as function of $a$ is meromorphic in all 
Riemann sphere, therefore take all values in 
$\widehat{\mathbb{C}}-\{0,1,\infty\}$ when $a^4\neq 1$ and $a\neq 0$. 
Then all elliptic curves are isomorphic to one elliptic curve in the 
form

\begin{equation}\label{jacobi2}
 y^2=(x^2-a^2)\left(a^2x^2-1\right),\quad \text{with}\ a^{4}\neq 1 
 \ \text{and}\ a\neq 0.
\end{equation}

The projection in the first coordinate is a meromorphic function of
degree 2 whose critical values are $\{a,-1/a,-a,1/a\}$. Note that
(\ref{jacobi2}) can be reduced to the curve:

\begin{equation}\label{jacobi-form}
 y^2=(x^2-1)\left(k^2x^2-1\right),\quad \text{with}\  
 k^2\neq \pm 1\ \text{and}\ k\neq 0,
\end{equation}

\noindent where $k=a^2$, this can be achieved applying the
transformation $\mu(z)=z/a$ to the points $\{a,-1/a,-a,1/a\}$. The
equation (\ref{jacobi-form}) is called a Jacobi normal form of an
elliptic curve. Therefore every elliptic curve is isomorphic to one
elliptic curve in the Jacobi normal form. In fact this elliptic curve
is isomorphic to the normal form defined by Edwards, see
\cite{Edwards}. The Edwards normal form is by definition:

\begin{equation}\label{edwards-form}
 x^2+y^2=a^2+a^2x^2y^2,\quad \text{with}\ a^{4}\neq 1 
 \ \text{and}\ a\neq 0,
\end{equation} 
in this case the projection to the first coordinate is of degree 2 and 
ramifies exactly on $\{a,-1/a,-a,1/a\}$ (see Figure 
\ref{symmetric-points}). 

\begin{figure}
\begin{center}
 \includegraphics[width=10cm]{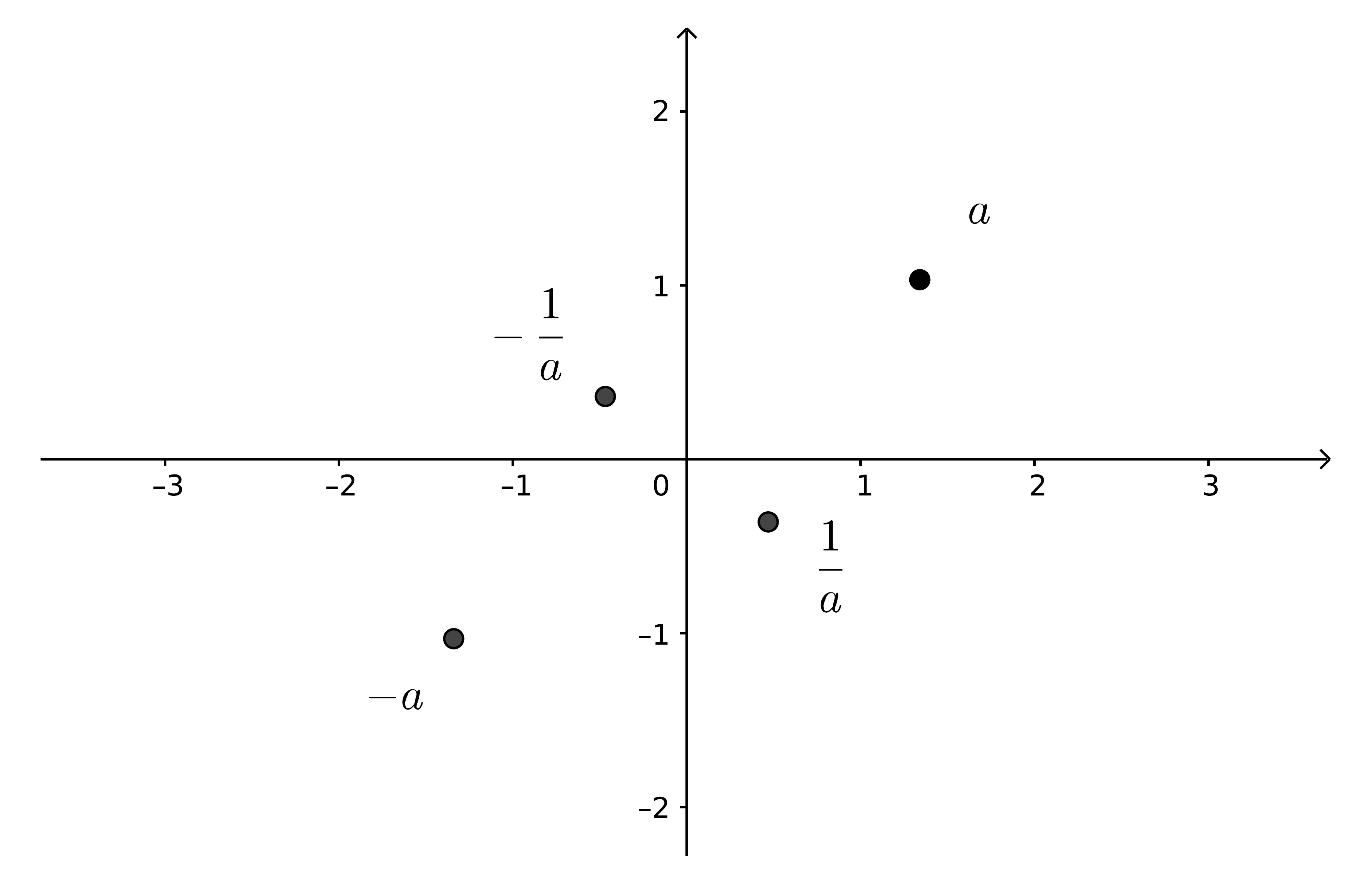}
\caption{A configuration of points for the Edwards normal 
form with $a=1.34023+1.032 i$.} \label{symmetric-points} \end{center} 
\end{figure} 

\par We summarize the previous discussion in the following theorem:

\begin{theorem}\label{Jacobi-Edwards-equivalence}
	Every elliptic curve $X$ is isomorphic to one elliptic curve in the 
	form
		
	\begin{equation*}
		y^2=(x^2-a^2)\left(a^2x^2-1\right),\quad \text{with}\ a^{4}\neq 1 
		\ \text{and}\ a\neq 0,
	\end{equation*}
	
	\noindent which is isomorphic to the elliptic curve in the Edwards 
	normal form (\ref{edwards-form}) and isomorphic to the elliptic 
	curve in the Jacobi normal form (\ref{jacobi-form}) with $k=a^{2}$.
\end{theorem}
For different approach of the above equivalence consult \cite{Edwards}.

\section{The $J$ invariant and the Hesse normal 
form}\label{Section-Hesse} We say that $\lambda$ and $\lambda'\in 
\widehat{\mathbb{C}}-\{0,1,\infty\}$ are equivalent if 
$\{0,1,\lambda,\infty\}$ and $\{0,1,\lambda',\infty\}$ are Möbius 
equivalent, it occurs if and only if there is a Möbius transformation 
$\mu$ such that permutes $\{0,1,\infty\}$ and satisfies 
$\mu(\lambda)=\lambda'$. Computing these cross ratios, we obtain that 
the points equivalent to $\lambda$ are:

\begin{equation}\label{equivalent-cross-ratios}
 \lambda,\quad \frac{1}{\lambda},\quad \lambda-1,\quad
 \frac{1}{\lambda-1},\quad \frac{\lambda}{\lambda-1},\quad
 \frac{\lambda-1}{\lambda}.
\end{equation}

Define the rational function

\begin{equation}\label{j-invariant}
 J(\lambda)=\frac{(\lambda^2-\lambda+1)^3}{\lambda^2(\lambda-1)^2},
\end{equation}

\noindent this is a rational map of degree 6 called the $J$
invariant, it is an invariant because $\lambda$ is equivalent to
$\lambda'$ if and only if $J(\lambda)=J(\lambda')$ (our definition of
$J$ is according to  \cite[\S 6.3.2]{Donaldson}, it differs from the
definition of other authors by some factor). Then, the $J$ invariant
is a function well-defined in $\mathcal{M}$ (viewed as four points on
the sphere modulo Möbius transformations), the value of $J$ in a set
of  four points $\{z_1, z_2, z_3, z_4 \}$ is obtained by
$J(\lambda)$, where $\{0,1,\lambda,\infty\}$ is a 4-point set
equivalent to $\{z_1, z_2,  z_3, z_4 \}$.

\par In the following diagram we represent the critical values and
the critical points of $J$, the weight of the arrows represent the
multiplicity of the critical point:

\begin{equation}\label{branching-diagram}
 \xymatrix{ 0\ar^2[rd] & 1\ar^2[d] & \infty\ar_2[dl] &
 2\ar^2[rd] & \frac{1}{2}\ar^2[d] & -1\ar_2[ld] &
 -\rho\ar^3[d]  & -\rho^2\ar_3[ld] \\ & \infty & & &
 \frac{3^3}{4} & & 0 & }
\end{equation}

\noindent where $\rho$ is the cubic root $\exp(2\pi i/3)$. Note that
the Fermat cubic is a critical point of $J$, which corresponds to an
equilateral triangle. The $J$ invariant can be factorized by the
following rational functions (see \cite[p. 93]{Donaldson}):

\begin{equation}\label{decomposition}
 \widehat{\mathbb{C}}\xrightarrow{\frac{\rho z+\rho^2}{z+\rho^2}}
 \widehat{\mathbb{C}}\xrightarrow{z^3+\frac{1}{z^3}}
 \widehat{\mathbb{C}} \xrightarrow{-\frac{3^3}{z-2}}
 \widehat{\mathbb{C}},
\end{equation}

\par \noindent (the constant $-3^3$ in the last function in
(\ref{decomposition}) differs from the constant that appear in
\cite[p. 93]{Donaldson}, which seems to contain a mistake).

\par Consider the elliptic curve in the form

\begin{equation}
 x^3+y^3+1=3kxy, \quad k\in \mathbb{C}\ \text{with}\ k^3\neq 1,
\end{equation}
called the Hesse normal form. As in Example \ref{example-Fermat} we take 
the function $x+y$, this is also of degree two with critical values on 
$-k$ and on the roots of the cubic polynomial

\begin{equation}
 z^3-3kz^2+4
\end{equation}

\noindent this polynomial has discriminant $\Delta = 4^23^3(k^3-1)$,
then by our hypothesis there are three distinct roots. So we can take
the value of the $J$ invariant in the set of critical values,
obtaining a function $\varphi(k)$ depending on $k$, and well defined
in $\widehat{\mathbb{C}}-\{1,\rho,\rho^2,\infty\}$. Since the roots
of a polynomial depends continuously of the coefficients,
$\varphi(k)$ is a continuous function on
$\widehat{\mathbb{C}}-\{1,\rho,\rho^2,\infty\}$. And it is also
analytic in $\mathbb{C}-\{0,1,\rho,\rho^2\}$, as we can see looking
at the general formulas of cubic equations (taking locally suitable
branches of the roots) and taking the formula (\ref{cross-ratio})
with $z_4=-k$:

\begin{equation}\label{z-nu}
 z_\nu=k-\rho^\nu\alpha-\frac{1}{\rho^\nu}\frac{k^2}{\alpha},\quad
 \nu=1,2,3,
\quad \text{where}\ \alpha= \sqrt[3]{2-k^3+2i\sqrt{k^3-1}}.
\end{equation}

\par Let us analyse what is happening in $\infty$ and in the cubic
roots of 1. If $k_n$ is a sequence such that $k_n\to \rho^\nu$, again
as the roots depend continuously of the coefficients, for each $n$,
we can choose a labelling of the roots of $z^2-3k_nz^2+4$ such that
the cross ratios tends to $\infty$, then $\varphi(k_n)\to \infty$.

\par On the other hand, if $k_n\to \infty$, we affirm that
$\varphi(k_n)\to \infty$. Here is a sketch of the proof: considering
the equation (\ref{z-nu}), we have that $\alpha/k$ tends to
$-1,-\rho$ or $-\rho^2$, when $k\to \infty$, depending of the choice
of the branch of the cubic root in $\alpha$. With this, it is
straightforward to check that $\chi(z_1,z_2,z_3,-k)$ tends to $0,1$
or $\infty$, when $k\to \infty$, according to the choice of the
branch. Then, if we take for each $n$, $k=k_n$ in (\ref{z-nu}) we
have that the cross ratios $\chi(z_1,z_2,z_3,-k_n)$ accumulate on an
subset of $\{0, 1,\infty\}$ when $k_n\to \infty$. In any case $J$
tends to $\infty$.

\par Then $\varphi(k)$ is a continuous function on all the Riemann
sphere. Hence a rational function, whose poles are $\infty$ and the
cubic roots of 1.

\par Hence, for every $\lambda \in \widehat{\mathbb{C}}-\{0,1,\infty\}$
there exists $k\in \mathbb{C}$ with $k^3\neq 1$, such that
$\varphi(k)=J(\lambda)$, then $\lambda$ is equivalent to one cross
ratio of the set formed by $-k$ and the roots of $z^3-3k z^2+4$. 

\par We summarize the previous discussion in the following theorem:

\begin{theorem}\label{Theorem-Hesse}
 Every elliptic curve $X$ is isomorphic to one elliptic curve given in 
 the Hesse normal form: 
 
 \begin{equation}\label{hesse-normal-form} x^3+y^3+1=3kxy, \quad k\in 
 \mathbb{C}\ \text{with}\ k^3\neq 1,
 \end{equation} 
 
 \noindent which is isomorphic to the elliptic curve in the form
 
 \begin{equation}\label{new-form}
  y^2=(x+k)(x^3-3kx^2+4),\quad k\in \mathbb{C}\ \text{with}\ k^3\neq 1. 
 \end{equation} 
 \end{theorem} 

\begin{figure}
 \begin{center}
  \includegraphics[width=12cm]{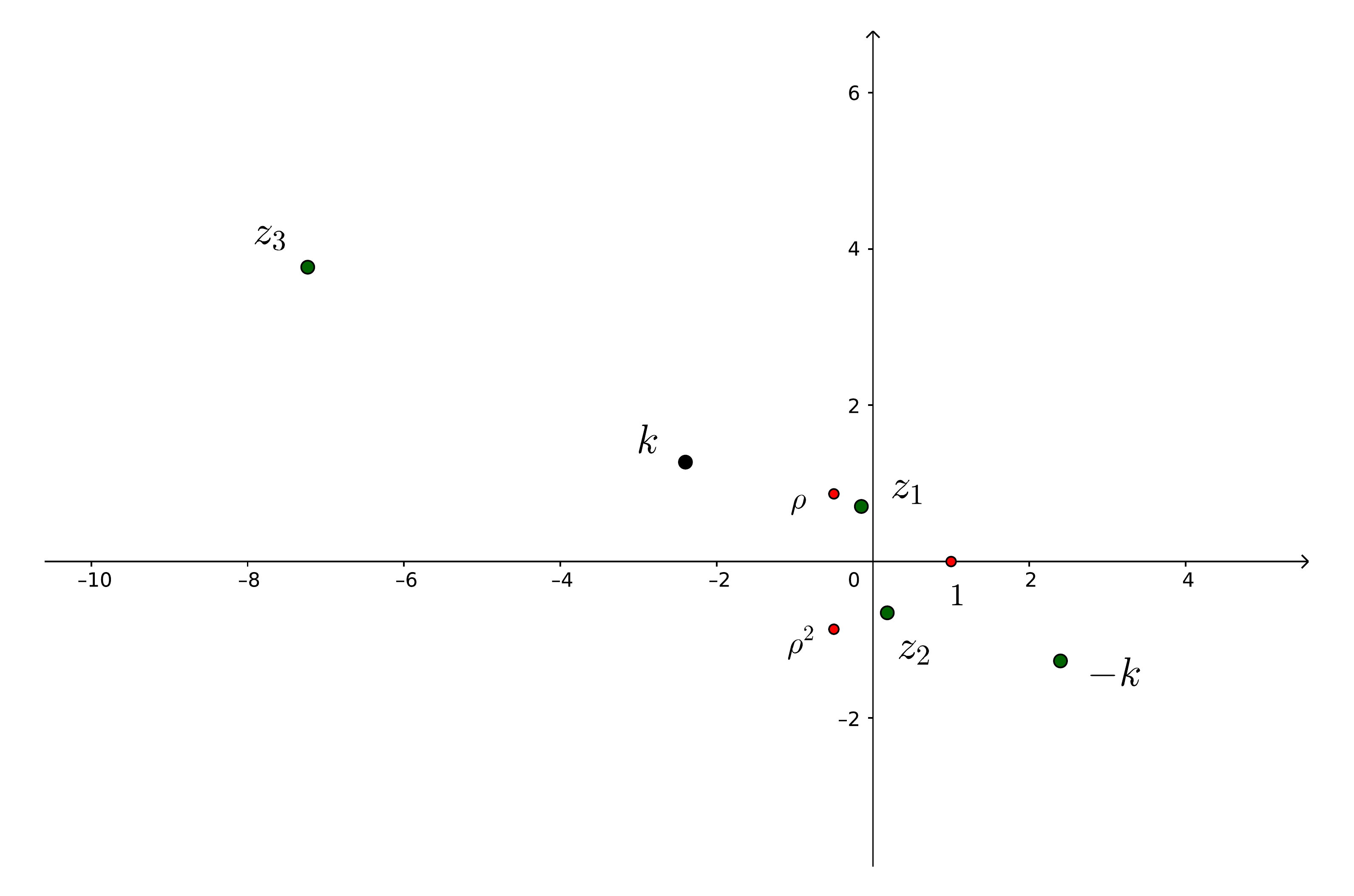}
\caption{A configuration of points for the Hesse normal form with
$k=-2.39882+1.27189 i$.}
\label{configuration-Hesse}
\end{center}
\end{figure} 

For different approaches consult \cite{Bonifant-Milnor}, \cite{Frium} 
and \cite{Popescu}. According to the formula obtained in \cite{Popescu} 
or in \cite{Bonifant-Milnor} the rational function $\varphi(k)$ should 
be:

\begin{equation}
\varphi(k)=\frac{27}{4}\left(\frac{k(k^3+8)}{4(k^3-1)}\right)^3. 
\end{equation}

\FloatBarrier


\begin{thebibliography}{20}

\bibitem{Forster} O. Forster, \emph{Lectures on Riemann surfaces},
Graduate Texts in Mathematics, 81, Springer-Verlag, New York-Berlin,
1981.

\bibitem{Donaldson} S. Donaldson, \emph{Riemann surfaces}, Oxford
Graduate Texts in Mathematics, 22,  Oxford University Press, Oxford,
2011.

\bibitem{Edwards} H. Edwards, \emph{A normal form for elliptic
curves}, Bull. Amer. Math. Soc. (N.S.) 44 (2007), no. 3, 393–422.

\bibitem{Bonifant-Milnor} A. Bonifant \& J. Milnor, \emph{On Real and
complex cubic curves}, preprint arXiv:1603.09018v2 [math.AG].

\bibitem{Frium} H. R. Frium, \emph{The group law on elliptic curves
in Hesse form}, Finite fields with applications to coding theory,
cryptography and related areas (Oaxaca, 2001), 123–151, Springer,
Berlin, 2002.

\bibitem{Popescu} P. Popescu-Pampu, \emph{Iterating the
Hessian: A dynamical system on the moduli space of elliptic curves
and dessins d’enfants}. Noncommutativity and singularities, 83–98,
Adv. Stud. Pure Math., 55, Math. Soc. Japan, Tokyo, 2009.

\bibitem{Steinmetz} N. Steinmetz, \emph{Nevanlinna theory, Normal
families, and algebraic differential equations}, Universitext,
Springer, Cham (2017).

\end{thebibliography}
\end{document}